\documentclass{amsart}
\title[On the integrability of the shift map]{on the integrability of the shift map on twisted pentagram spirals}
\author{Gloria Mar\'i Beffa}
\thanks{This paper is supported by the author NSF grant DMS \#0804541 and \#1405722}
\address{Mathematics Department, University of Wisconsin-Madison, WI 53706\\ \rm{maribeff@math.wisc.edu}}
\usepackage{amsmath,amsfonts,amsthm}
\usepackage{graphicx,pstricks}
\usepackage{color}
\newtheorem{theorem}{Theorem}
\numberwithin{theorem}{section}

\newtheorem{lemma}[theorem]{Lemma}
\newtheorem{comment}[theorem]{Comment}

\theoremstyle{definition}
\newtheorem{definition}[theorem]{Definition}

\def\RP{\mathbb {RP}}
\def\R{\mathbb R}

\def\Smap{\mathcal S}

\def\G{\mathcal{G}}
\def\C{\mathcal{C}}
\def\TS{\mathrm{TS}}
\def\PSL{\mathrm{PSL}}

\def\A{\mathcal{A}}
\def\B{\mathcal{B}}

\def\T{\overline{T}}
\def\V{\widetilde{V}}
\def\tc{\tilde{c}}
\def\tb{\tilde{b}}

\def\l{\lambda}
\def\hl{\widehat{\l}}

\def\be{\beta}

\def\b{{b}}
\def\cg{{c}}
\def\ag{{a}}
\def\ha{\widehat{\ag}}
\def\hb{\widehat{\b}}
\def\hc{\widehat{c}}
\def\hV{\widehat{V}}
\def\tV{\widetilde{V}}

\def\al{\alpha}
\def\be{\beta}

\begin{document}
\maketitle
\begin{abstract} In this paper we prove that the shift map defined on the moduli space of twisted pentagram spirals of type $(N, 1)$ possesses a non-standard Lax representation with an associated monodromy whose conjugation class is preserved by the map. We prove this by finding a coordinate system in the moduli space of twisted spirals, writing the map in terms of the coordinates and associating a natural parameter-free non-standard Lax representation. We then show that the map is invariant under the action of a $1$-parameter group on the moduli space of twisted $(N,1)$ spirals, which allows us to construct the Lax pair. We also show that the  monodromy defines an associated Riemann surface  that is preserved by the map. We use this fact to generate invariants of the shift map. 
%We also describe the relationship between spiral and pentagram maps and give a different proof of a result by Schwartz (\cite{Sspiral}) stating that the spiral is the $N$-root of the pentagram map, explaining also why in certain coordinates the pentagram map appears to be non-local (even though it clearly is).
 \end{abstract}
\section{Introduction}
The pentagram map is defined on planar, convex $N$-gons. The map $T$ takes a polygon with vertices $p_k$ to the polygon with vertices formed by intersecting two segments: one created by joining the vertices to the right and to the left of the original one, $\overline{p_{k-1}p_{k+1}}$, the second one by joining the original vertex to the second vertex to its right $\overline{p_kp_{k+2}}$ (see Fig. 1). These newly found vertices form a new $N$-gon. The pentagram map takes the first $N$-gon to this newly formed one. (The name pentagram map comes from the star formed in Fig. 1 when applied to pentagons.) As surprisingly simple as this map is, it has an astonishingly large number of properties.

\vskip 2ex
\centerline{\includegraphics[height=2in]{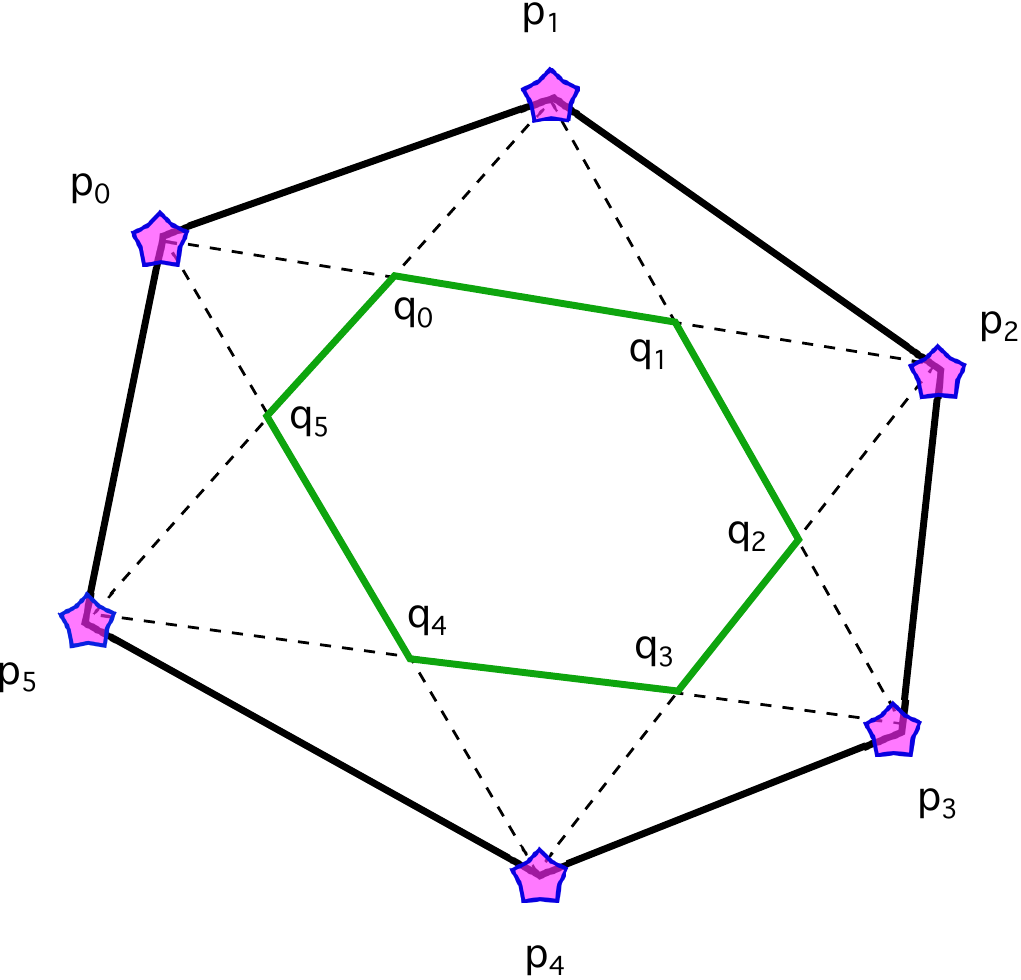}}
\vskip 1ex
\centerline{Figure 1: the pentagram map on hexagons}
\vskip 2ex

 It is a classical fact that if $P$ is a pentagon, then $T(P)$ is projectively equivalent to $P$. It also seems to be classical that if $P$ is a hexagon, then $T^2(P)$ is projectively equivalent to $P$ as well. The constructions performed to define the pentagram map can be equally carried out on the projective plane so we assume this is where the polygons live. When defined on the moduli space of pentagons (the set of equivalence classes up to the projective action, as described by the projective invariants of the polygons) { $T$ is the identity, while it is an involution when defined on the moduli space of hexagons}. In general, one should not expect to obtain a closed orbit for any $N$; in fact orbits exhibit a quasi-periodic behavior classically associated to completely integrable systems. This was conjectured in \cite{S3}.

 The author of \cite{S3} defined the pentagram map on what he called  {\it $N$-twisted polygons}, that is, infinite polygons with vertices $p_k$, for which $p_{N+k} = M(p_k)$ for all $k$, where $M$ is the {\it monodromy}, a projective automorphism of $\RP^2$. 
Following a somehow dormant period the pentagram map has come back with full force after the publication of \cite{OST} where the authors proved that the map on twisted polygons  is not only completely integrable, but its continuous limit is the Boussinesq equation, a well known completely integrable PDE modeling certain dynamics of waves. A large number of publications have followed this, proving the integrability of the map on closed polygons (\cite{OTS2}, \cite{FS}), defining and proving integrability of generalizations (\cite{GSTV}, \cite{KS}, \cite{MB1},  \cite{MB2}), establishing connections to cluster algebras (\cite{Glick}) and more. During the last year Schwarz defined two new maps, the heat map on closed polygons (\cite{heat}) and the shift map on pentagram spirals (\cite{Schspiral}). Here we will focus on pentagram spirals and their shift maps.

In \cite{Schspiral} Schwartz defined what he called a pentagram spiral, a family of bi-infinite polygons in the projective plane that spiral inside and outside of themselves following a pentagram map-type of construction. The spiral is determined by a {\it seed}, vertices of a closed polygon together with a number of points on the sides of the closed polygon that mark the moment when the polygons start spiraling (see fig. 2, the stars mark the closed polygon, squares are the points on the sides). Spirals are classified by two numbers $(N,k)$ where $k$ is the number of  points on the sides in the $N$-closed polygon. Fig. 2 shows a $(5,2)$ spiral. A pentagram spiral has at most one seed point per side, and $k$ marks the number of spiraling branches that the pentagram spiral has.
\vskip 3ex
\centerline{\includegraphics[height=3.1in]{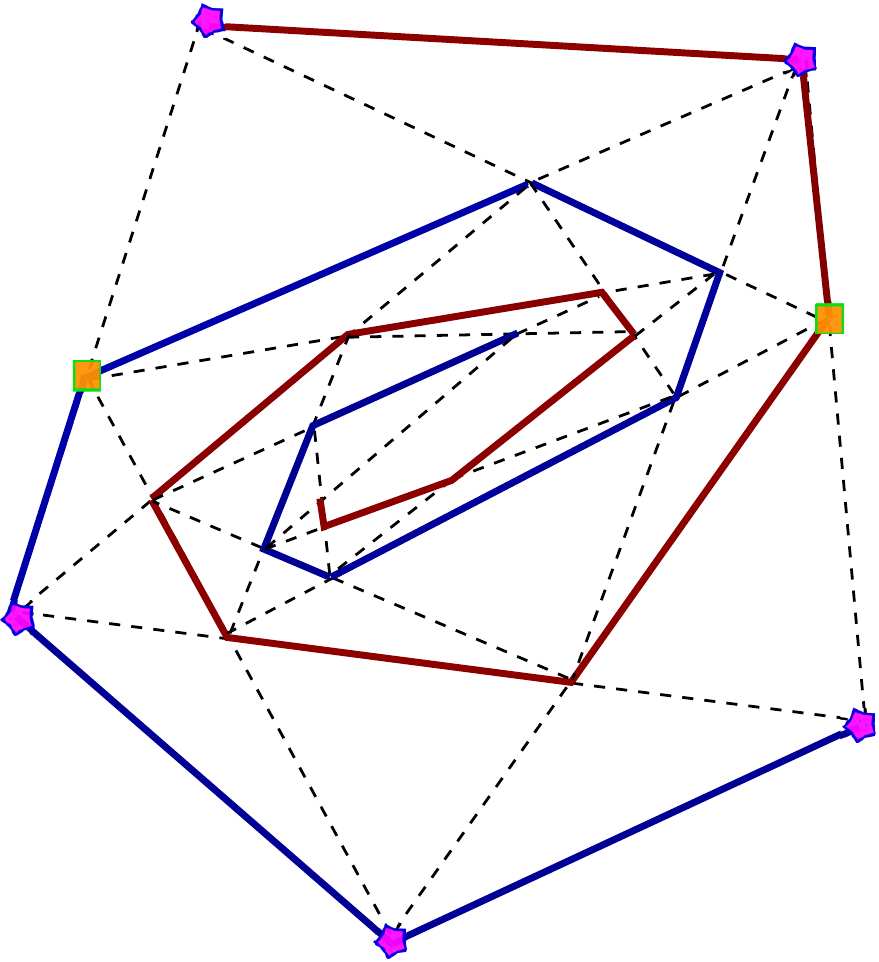}}
\centerline{Figure 2: a pentagram spiral of type $(5,2)$}
\vskip 2ex
In \cite{Schspiral} Schwartz studied the shift map on pentagram spirals, the map that assigns to each spiral the one obtained by shifting the vertices (and hence the seed) once forward, forward for us being the direction towards the interior of the polygon. He proved that such a map can be thought of, in a certain sense, as the $(N+1)$th root of the original pentagram map, and he also conjectured that, like the pentagram map, the shift map is also completely integrable.

In this paper we study the case of twisted $(N,1)$ spirals, denoted by $\TS(N,1)$; these are spirals where a monodromy map is applied each time the pentagram map acts after a full period $N+1$. The general case $(N,k)$ is now in progress. The paper is divided into several parts. In the first part we find a generating family of projective invariants for generic elements in $\TS(N,1)$, and we use them to define a coordinate system in the moduli space of $\TS(N,1)$.  We also use them to describe a parameter-free non-standard Lax representation for the shift map. By non-standard we mean that the boundary conditions are not periodic. The invariants will be found similarly to those of twisted polygons: we choose an appropriate lift of the spiral to $\R^3$ by imposing a number of normalizations. The lift defines a discrete moving frame for the spiral (in the sense of \cite{MMW}) and it provides us with a complete set of generating and independent projective invariants for twisted spirals which define a coordinate system in the moduli space. The lift exists only if $N\ne 3s+1$ for any $s$, an assumption we make from then on.  This is done in section 3, with theorem \ref{invariants} describing the coordinate system (the proof of this Theorem is too technical and appears in the Appendix). We then notice that once we shift the spiral the invariants will not merely be shifted. Indeed, the lift for the shifted spiral will have different normalization conditions and hence it will be different from the original lift (this is expected since the shifted spiral has shifted seeds).

In lemma \ref{lemma1} we prove that there are two proportions $\al$ and $\be$, determined by the two changes of seed when shifted - at the beginning and the end of the polygon -, such that the shifted lift equals the original lift times certain powers of $\al$ and $\be$ (again, the proof of this lemma appears in the Appendix). The powers  depend on the vertex and they recur every three vertices, except for the last ones. Lemma \ref{one} shows that the generating invariants also transform by shifting and multiplying by powers of $\al$ and $\be$. This is true for most invariants with the exception of those associated to the end seed point, the vertex where the polygon starts to spiral. Those end points need to be treated carefully. 
 Lemma \ref{one} and the proof of theorem \ref{scalingth} gives an explicit formula for the shift map in these coordinates.  (We also describe, towards the end, a different set of coordinates for which the shift map  is simply a shift of coordinates, except for the coordinates of the last vertex for which the map is highly complicated.)
 
 The last step is to prove that $\al$ and $\be$ are invariant under a certain $1$-parameter group action on the moduli space defined through scaling of the invariants (the same scaling used for the pentagram map) and to use this to show that the shift map is left invariant by that action. This is done in Theorem \ref{scalingth}. Introducing the scaling in the parameter-free Lax representation produces a valid non-standard Lax pair that can be used to generate invariants.  This and the  description of the preserved quantities are given in our last section, where, as an example, we generate invariants of the map for $N = 5$.

The author would like to thank R. Schwarz for continuous conversations. This paper has been  supported by NSF grants DMS \#0804541 and \#1405722.

\section{Background}
\subsection{The pentagram map and pentagram spirals} The pentagram map was originally defined by the author of \cite{S1} as a map defined on the space of closed $N$-polygons in the projective plane
\[
T: \C(N) \to \C(N).
\]
Given a closed $N$-polygon $\{p_i\}_{n=0}^{N-1}$, $p_i\in \RP^2$, we define the image of this polygon by the pentagram  map as $T(\{p_i\}) = \{q_i\}$ where \[q_i = \overline{p_{i-1}p_{i+1}}\cap\overline{p_ip_{i+2}}\] where $p_{N+k} = p_k$ (See Fig. 1.)  The map can equally be thought to act on the vertices of the polygons as $T(p_i) = q_i$. 
\begin{definition} (Twisted polygon) Let $P=\{p_i\}$ be an infinite polygon in the projective plane. We say that $P$ is an  {\it $N$-twisted polygon}, whenever there is a projective transformation $M$ such that
\[
p_{N+i} = M(p_i)
\]
for all $i$. $M$ is called the {\it monodromy}.
\end{definition}
One can easily check that the pentagram map is well defined on the space of twisted polygons. 

The idea of a pentagram spiral is built upon a choice of points on the sides of a closed polygon from which the polygon starts spiraling inside itself using pentagram transformations of consecutive vertices. We need merely one side point to start creating a branch of the spiral (see fig. 3), but one can, in principle, choose one point in several sides, creating the branching of several spirals, each one created using the pentagram image of the previous branch (see fig. 2). The vertices of the original polygon upon which the construction is built, together with the points on the sides form what the author of \cite{Schspiral} called {\it the seed} of the spiral. See \cite{Schspiral} for more details. A spiral built with the seed of an $N$-gon with $k$ distinguished points on $k$ different sides is called a pentagram spiral of type  $(N,k)$. In this paper we will focus on spirals of type $(N,1)$.

 \begin{definition} ($(N,1)$ Pentagram spiral) Given an N-polygon in $\RP^2$, $\{p_1, \dots, p_N\}$, and a point $p_{N+1}$ {\it on the side joining $p_N$ and $p_1$}, we define the $(N,1)$ pentagram spiral associated to the seed $\{p_1, \dots, p_N; p_{N+1}\}$ as the bi-infinite polygon with ordered vertices 
 \begin{equation}\label{spiral}
 \{, \dots, T^{-1}(p_{N-1}), T^{-1}(p_N), T^{-1}(p_{N+1}), p_1, p_2, \dots, p_N, p_{N+1}, T(p_1), T(p_2), \dots \}.
 \end{equation}
Let us call $p_{N+i}= T(p_{i-1})$ and $p_{-i} = T^{-1}(p_{N-i+1})$, for $i\ge1$ so that $p_{N+1} = T(p_0)$. Figure 3 shows a standard $(6,1)$-pentagram spiral.
 \end{definition}
 
\centerline{\includegraphics[height=4in]{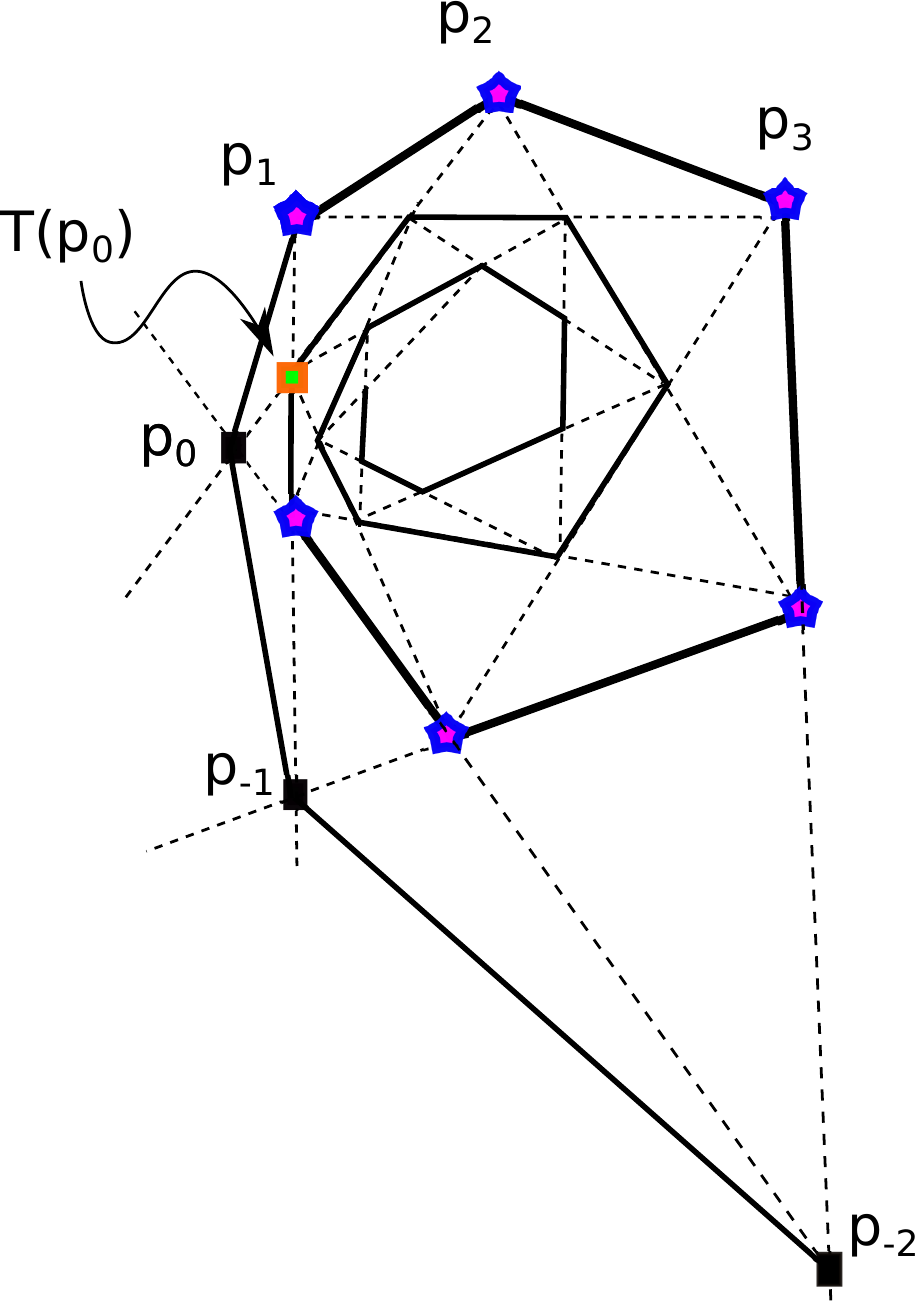}}

\centerline{ Figure 3: a pentagram spiral of type $(6,1)$}

\vskip 2ex
\begin{comment}{\rm Notice that from 
\[
p_{N+2} = T(p_1) = \overline{p_0p_2}\cap \overline{p_1p_3}, \quad\quad p_0 = T^{-1}(p_{N+1})= \overline{p_{N-1}p_N}\cap\overline{p_{N+1}p_{N+2}}
\]
one might conclude that the spiral is not well-defined. To resolve this problem we note that
\[
T(p_1) = \overline{p_{N+1}p_2}\cap \overline{p_1p_3} = \overline{p_0p_2}\cap \overline{p_1p_3}
\]
the equality being clear from Figure 3.
}
\end{comment}
The moduli space of $(N,1)$ pentagram spirals was proven to be $2N-7$ dimensional in \cite{Schspiral}. In order to facilitate the creation of a Lax representation for the shift map on spirals we will consider {\it twisted} pentagram spirals by introducing a monodromy. The author of \cite{Schspiral} also defined the concept of twisted spiral, with a more abstract approach, and proved that its moduli space was $2N+1$ dimensional. Although described differently, both concepts coincide.
\begin{definition} (Twisted pentagram spirals) Given an N-polygon in the projective plane $\RP^2$, $\{p_1, \dots, p_N\}$, an element of the projective group $M\in \PSL(3,\R)$, and a point $p_{N+1}$ in the segment joining $p_N$ and $M\cdot p_1$,  we define the {\it twisted pentagram spiral} associated to the seed $\{p_1, \dots, p_N; p_{N+1}\}$ with monodromy $M$ as the infinite polygon with ordered vertices 
\begin{multline}\label{twistedspiral}
\{, \dots, M^{-1}\cdot T^{-1}(p_{N-1}), M^{-1}\cdot T^{-1}(p_N), p_0, p_1, p_2, \dots, p_N, M\cdot T(p_0), \\ M\cdot T(p_1), M\cdot T(p_2),  \dots\},
\end{multline} where $p_{N+1} = M\cdot T(p_0)$, $p_{N+i}= M\cdot T(p_{i-1})$, $p_{-i} = M^{-1}\cdot T^{-1}(p_{N-i+1})$,  for $i\ge 1$, and where the {\it monodromy} $M$ acts each time a period $N+1$ is completed and $T$ is applied. 
\end{definition}
Next we will prove that the moduli space of twisted pentagram spirals is a space of dimension $2N+1$ and we will describe a generating set of invariants that will define coordinates in it. 

\section{ The moduli space and the shift map}
\subsection{A coordinate system for the moduli space of twisted pentagram spirals} 
The moduli space of twisted $N$-polygons in $\RP^2$ has been well studied in \cite{OST}, where the authors proved that the space has dimension $2N$. They also described a coordinate system defined by a basis of projective discrete invariants of polygons. It is defined as follows: 

Assume that we have a twisted $N$-polygon $\{p_i\}$, with a monodromy $M$. One can prove (\cite{OST}) that if $N\ne 3s$ for all $s$, then there exist unique lifts of $p_i$ to $\R^3$, call them $V_i$, such that
\[
\det(V_i, V_{i+1}, V_{i+2}) = 1
\]
for all $i$.  Under these conditions, one can always find invariants $e_i$, $f_i$ satisfying the relation
\begin{equation}\label{penta}
V_{i+3} = e_i V_{i+2}+f_i V_{i+1}+V_i
\end{equation}
for all $i$, where $e_i, f_i$ are functions of the vertices of the polygon. The functions $e_i$, $f_i$, $i=1,2,\dots, N$, define a coordinate system for their moduli space since they determine the polygon up to the projective action of $\PSL(3,\R)$ (which is linear on $V_i$). In this section  we define the analogous set of coordinates for pentagram spirals. The key to this definition and to the existence of the Lax pair is to choose the lifts of $T(p_i)$ and $T^{-1}(p_i)$ appropriately to guarantee scaling invariance of the shift map, and the existence of the Lax representation, which is our ultimate goal. If one chooses more straightforward lifts, scaling invariance is not guaranteed.

\begin{theorem}\label{lift} Let $P$ be a twisted pentagram spiral and assume that $N \ne 3s+1$. Then there exists a unique lift of the seed $\{p_1, p_2, \dots, p_N; p_{N+1}\}$ to $\{V_1, \dots, V_N; V_{N+1}\}$ and a unique lift of the spiral $P$ to the polygon in $\R^3$ with vertices $\{V_i\}_{-\infty}^{+\infty}$
\[
\dots, V_{-1}, V_0, V_1, V_2, \dots, V_N, V_{N+1}, V_{N+2},  \dots 
\]
such that:

\begin{equation}\label{TVlift}
V_{N+i} = M T(V_{i-1}) \quad\hbox{\rm with}\quad T(V_i) = \left( V_{i-1} \times V_{i+1}\right)\times \left(V_i\times V_{i+2}\right);
\end{equation}

\begin{equation}\label{TinvV}
V_{-i} = M^{-1} \T(V_{N-i+1})\quad\hbox{\rm with}\quad\T(V_i) = c_{i+1} \left(V_i\times V_{i+1}\right)\times\left(V_{i-2}\times V_{i-1}\right)
\end{equation}
where $c_i = \displaystyle\frac{\det(V_{i+1}, V_{i+2}, V_{i+3})}{\det(V_i, V_{i+1}, V_{i+2})}$; and 
\begin{equation}\label{normalization}
\det(V_i, V_{i+1}, V_{i+2}) = 1
\end{equation}
for $i=0,1,2,\dots,N$. 
\end{theorem}
\begin{comment}{\rm First of all, notice that we are abusing notation by using the letter $T$ for both the projective map and its lift. The domain should make clear which one is which. Notice also, that the lift of $T$ to $V_i$ is a convenient one chosen from an infinite number (any invariant multiple will do). Finally, as it will become clear later, $\T$ is not equal, but proportional, to the inverse of the lift $T$, with the proportion having an important role in the coordinate description of the shift map and on its scaling invariance.  Some other lifts of $T^{-1}$ also work, but the one chosen here will make our calculations the simplest.
}
\end{comment}
\begin{proof} 
Let $\V_k$ be any lift of  $p_i$, $i=0,1,\dots,N$, with $p_0 = T^{-1}(M^{-1}p_{N+1})$. Assume the lift we are looking for is of the form $V_k = \l_k \V_k$. From the definition of $T(V_i)$ in the statement of the theorem, it is clear that 
\[
T(V_i) = \l_{i-1}\l_i\l_{i+1}\l_{i+2} \widetilde{T(V_i)}
\]
where $\widetilde{T(V_i)} = \left( \V_{i-1} \times \V_{i+1}\right)\times \left(\V_i\times \V_{i+2}\right)$. Hence, and since $V_{N+k} = M T(V_{k-1})$, we can define
\begin{equation}\label{lambdas}
\l_{N+1} = \l_{-1}\l_0\l_1\l_2 \quad \hbox{\rm and}\quad \l_{N+k} = \l_{k-2}\l_{k-1}\l_k\l_{k+1} \quad\hbox{for all} ~k=2,\dots,
\end{equation}
where $\l_{-1}$ satisfies $V_{-1} = \l_{-1} \V_{-1}$, with $V_{-1} = M^{-1}\T(V_N)$ and $\V_{-1}=M^{-1}\widetilde{T(V_{N})}$. The tilde will always indicate that we are evaluating the map on the $\{\V_i\}$ lift. Let us find $\l_{-1}$ first. Since
\[
M V_{-1} = \T(V_N) = c_{N+1} \left(V_N\times V_{N+1}\right)\times\left(V_{N-2}\times V_{N-1}\right)
\]
from the definition of $c_i$ in the statement of the theorem, we have that
\[
c_{N+1} = \frac{\l_{N+4}}{\l_{N+1}} \tilde c_{N+1}.
\]
From here
\[
V_{-1} = \frac{\l_{N+4}}{\l_{N+1}} \l_N\l_{N+1}\l_{N-2}\l_{N-1}\widetilde{V_{-1}},
\]
and therefore
\begin{equation}\label{lm1}
\l_{-1} = \l_{N+4}\l_N\l_{N-1}\l_{N-2}.
\end{equation}
Now, condition (\ref{normalization}) results in the following equations
\[
\l_i\l_{i+1}\l_{i+2} = g_i
\]
where $g_i = \det(\V_i, \V_{i+1}, \V_{i+2})^{-1}$, for $i=0,1,\dots,N$. If we apply logarithms to both sides of these equation (adjusting for signs if needed), we get the system
\[
\Lambda_i+\Lambda_{i+1}+\Lambda_{i+2} = G_i
\]
where $\Lambda_i = \ln \l_i$; $G_i = \ln g_i$, $i=0,1,\dots, N$. Using (\ref{lambdas}) and (\ref{lm1}) we additionally have  
\[\Lambda_{N+k} = \Lambda_{k-2}+\Lambda_{k-1}+\Lambda_k+\Lambda_{k+1},
\] 
for $k=2,3,\dots$; and since $\l_{N+4} = \l_2\l_3\l_4\l_5$, we obtain
\begin{eqnarray*}\Lambda_{N+1} &=& \Lambda_{N+4}+\Lambda_{N}+\Lambda_{N-1}+\Lambda_{N-2}+\Lambda_0+\Lambda_1+\Lambda_2 \\&=& \Lambda_{N}+\Lambda_{N-1}+\Lambda_{N-2}+\Lambda_0+\Lambda_1+2\Lambda_2+\Lambda_3+\Lambda_4+\Lambda_5.  \end{eqnarray*} 

 The $(N+1)\times(N+1)$ coefficient matrix of this system is thus given by
\[
\left(\begin{array}{cccccccccccc} 1 & 1 & 1 & 0 & 0 & 0 & \dots & 0 & 0 & 0 & 0 & 0 \\ 0&1&1&1& 0& 0& \dots & 0 & 0& 0& 0& 0 \\ \vdots & \vdots & \ddots & \ddots & \ddots & \ddots & \ddots & \ddots & \ddots & \ddots &\vdots & \vdots \\ 0 & 0 & 0 & 0 & \dots & 0 & 0 & 0 & 0 & 1 & 1 & 1 \\ 1 & 1 & 2 & 1 & 1 & 1 & 0 & \dots & 0 & 1 & 2 & 2  \\ 2 & 2 & 3 & 2 & 1 & 1 & 0 & \dots & 0 & 1 & 1 & 2\end{array}\right)
\]
We need to calculate its determinant. If we use  rows one, four and $N-1$ to row reduce the last two rows, they become
\[
\left(\begin{array}{ccccccccc}0&0&1&0&0&\dots&0&1&1\\ 0&0&1&1&0&\dots&0&0&1\end{array}\right).
\]
This reduction allow us to remove the first two rows and columns of the matrix so the determinant of the coefficient matrix equal that of the $(N-1)\times(N-1)$ matrix
\[
\begin{pmatrix} 1&1&1&0&\dots&0&0&0\\ 0&1&1&1&\dots&0&0&0\\ \vdots&\vdots&\ddots&\ddots&\ddots&\ddots&\vdots\vdots\\ 0&\dots&0&0&0&1&1&1\\ 1&0&0&0&\dots&0&1&1\\ 1&1&0&0&0&\dots&0&1\end{pmatrix}.
\]
Using again the first two rows and row reduction we can change the last two rows to
\[
\left(\begin{array}{ccccccccc} 0&0&0&1&0&\dots&0&1&1\\ 0&0&-1&0&0&\dots&0&0&1\end{array}\right)
\]
which again allow us to remove the first two rows and columns and have the determinant  be equal to that of the $(N-3)\times(N-3)$ matrix 
\[
\begin{pmatrix} 1&1&1&0&\dots&0&0&0\\ 0&1&1&1&\dots&0&0&0\\ \vdots&\vdots&\ddots&\ddots&\ddots&\ddots&\vdots\vdots\\ 0&\dots&0&0&0&1&1&1\\ 0&1&0&0&\dots&0&1&1\\ -1&0&0&0&0&\dots&0&1\end{pmatrix}.
\]
Once this form of the matrix is achieved, the same process (using the first three rows to row reduce the last two) will produce an exact replica of the matrix with a smaller size. Each reduction process will reduce the size by $3$. Since we start with size $N-3$, reiterating the process we get to
\[
\begin{pmatrix} 1&1&1&0\\ 0&1&1&1\\ 0&1&1&1\\-1&0&0&1\end{pmatrix}, \quad \begin{pmatrix} 1&1&1&0&0\\ 0&1&1&1&0\\0&0&1&1&1\\ 0&1& 0&1&1\\ -1&0&0&0&1\end{pmatrix}, \quad\text{or}\quad
\begin{pmatrix} 1&1&1&0&0&0\\ 0&1&1&1&0&0\\ 0&0&1&1&1&0\\0&0&0&1&1&1\\0& 1&0&0&1&1\\ -1&0&0&0&0&1\end{pmatrix},
\]
in the cases $N=3s+1$, $N = 3s+2$ or $N = 3s$, respectively. The determinants of these matrices are $0$, $3$ and $3$, also respectively, therefore
the lift is unique if, and only if $N \ne 3s+1$, as stated.
\end{proof}

Given the lift $\{V_i\}$ we just found, let  $\{\ag_i, \b_i, \cg_i\}_{i=1}^\infty$ be defined by the relation
\begin{equation}\label{spiralrel}
V_{i+3} = \ag_i V_{i+2}+\b_iV_{i+1}+\cg_i V_i.
\end{equation}
for any $i$. (Notice that the definition of $c_i$ in the statement of the theorem \ref{lift} coincides with this one.) From now on we will assume that $N\ne 3s+1$ for any $s$.
\begin{theorem}\label{invariants} The moduli space of generic, strictly convex twisted spirals is a $2N+1$ manifold. A set of coordinates for a generic spiral is given by the set of discrete invariants $\G = \{\{\ag_i, \b_i\}_{i=0}^{N-1},  c_{N}\}$.
\end{theorem}

The proof of this theorem is rather long and mainly calculational, with a few interesting algebraic relations. One of the key parts of the proof is to show that if 
\[
K_i = \begin{pmatrix} 0&0& c_i\\ 1&0&b_i\\ 0&1&a_i\end{pmatrix}
\]
then
\[
 K_{N+i} = \A_i^{-1} K_{i-2}\A_{i+1}, \quad\quad K_{-i} = \B_{-i}^{-1} K_{N-i-1}\B_{-i+1}
\]
where $\A_i$ and $\B_{-i}$ are local and depend only on invariants nearby $p_{i-2}$ or $p_{N-i-1}$, respectively. This allow us to narrow the proof to the study of invariants that are close to the ends of the seed. The matrix $\A_1$ will have a crucial role in the description of the monodromy, as we will see towards the end of the paper. We have included the complete proof in the appendix.

 The last step in this section is writing the shift map in these coordinates.
\subsection{The shift map in coordinates}
 Let $\TS(N,1)$  be the manifold of twisted spirals described by the coordinates above in open generic subsets. Define the shift map 
\[
\Smap: \TS(N,1) \to \TS(N,1)
\]
as the map assigning to a given spiral,   the spiral obtained by shifting its vertices   once forward. That is
\begin{equation}\label{shift}
\Smap(\{p_1, p_2,  \dots, p_N; M\cdot T(p_0)\}) = \{p_2, \dots, p_N, M\cdot T(p_0); M\cdot T(p_1)\}. 
\end{equation}
The first of the following two  lemmas describes the interrelation between the lift of the spiral as given in Theorem \ref{lift}, and that of the shifted one.

Let   $A_i = c_i+a_ib_{i-1}$ be defined as in the Appendix, and let $\al$ and $\be$ be determined by the equations 
\begin{equation}\label{albe1}
\al^2\be = c_N ; \quad \al^{-1}\be^{-2} = \frac{A_3 A_0^2}{c_{-1}c_N} = \frac{A_3A_0}{A_1}
\end{equation}
if $N = 3s+2$; and
\begin{equation}\label{albe2}
\al^{-1}\be = c_N ; \quad \al^{-2}\be^{-1} = \frac{A_3 A_0^2}{c_{-1}c_N} = \frac{A_3A_0}{A_1}
\end{equation}
if $N = 3s$.

\begin{lemma} \label{lemma1}Let $\{p_1,\dots,p_N; M\cdot T(p_0)\}$ be a twisted spiral, and let $\{V_i\}$ be the lift described in theorem \ref{lift}. Let $\{p_2, \dots, p_N, M\cdot T(p_0); M\cdot T(p_1)\}$ be its shift and let $\{\hV_i\}$ be its analogous lift (the index matches that of $p_i$). Then
\begin{enumerate}
\item  if $N = 3s+2$, 
\begin{equation}\label{al1}
\hV_0 = \al^{-1} \be^{-1}V_0; \quad \hV_{3r+1} = \al^{-1}\be^{-1}V_{3r+1}; \quad \hV_{3r+2} = \al V_{3r+2}; \quad \hV_{3r} = \be V_{3r}
\end{equation}
\noindent with subindices ranging from $0$ to $N+2$, and
 \begin{equation}\label{al12}
 \hV_{N+3} = \al^{-1}\be^{-1}V_{N+3},\quad \hV_{N+4} = \al V_{N+4};
 \end{equation}
\item if $N = 3s$,
\begin{equation}\label{al2}
\hV_0 = \al^{-1}\be^{-1}V_0; \quad \hV_{3r+1} = \be V_{3r+1}; \quad \hV_{3r+2} = \al^{-1}\be^{-1}V_{3r+2}; \quad \hV_{3r} = \al V_{3r}
\end{equation}
\noindent with subindices ranging from $0$ to $N+2$, and 
\begin{equation}\label{al22}
\hV_{N+3} = \be V_{N+3},\quad \hV_{N+4} = \al^{-1}\be^{-1}V_{N+4}. 
\end{equation}
\end{enumerate}
\end{lemma}

Once more the proof of this lemma is rather long, based on mainly linear algebra and a careful analysis of the different cases. We have included it in the Appendix. Only to remark that an interesting result appearing in the proof is the algebraic {\it description of $\al\be^2$ as measuring the failure of $\T$ to be the twisted inverse of $T$ at $V_0$}. That is
\[\begin{array}{cccc}
V_0 &=& \al\be^2 M^{-1}\T(MT(V_0))~~ &\hbox{if} ~ N=3s+2;\\\\ V_0 &=& \al^{-2}\be^{-1} M^{-1}\T(MT(V_0))~~&\hbox{if}~N=3s.
\end{array}\]
The factors $\al$ and $\be$ can be written in terms of the generators by finding $c_{-1}$ and $A_0$ as functions of $\{a_i,b_i\}_{i=0}^{N-1}$, $c_N$. Indeed,
 using (\ref{cm1}), (\ref{rel}) and (\ref{rel2}) in the appendix. If $B_i = c_i+b_i a_{i-2}$, so that $B_N = c_N + b_N a_{N-2}$, we get
\[
\frac{c_{-1}}{A_0} = \frac{c_N}{B_NA_2}; \quad B_N = \frac{c_N^2}{A_1 A_2}
\]
and so
\[
\frac{c_{-1}}{A_0}=\frac{A_1}{c_N},\hskip2ex \text{and}\hskip2ex\frac{A_3A_0}{A_1} = \frac{A_3A_0^2}{c_{-1}c_N} = \frac{A_3 a_{N-1}}{A_1^2 a_0}c_N\hskip2ex \text{so that}\hskip2ex A_0=\frac{a_{N-1}c_N}{A_1a_0}.
\]
Our next lemma relates the invariants for the different spirals.
\begin{lemma}\label{one}
Assume a twisted pentagram spiral has a lift $\{V_i\}$ and the shifted spiral has a lift $\{\hV_i\}$ as in theorem \ref{lift}. Let $\{\ag_i, \b_i\}_{i=0}^{N-1}\cup \{c_N\}$ be the invariants defined by the lift in theorem \ref{lift} while $\{\ha_i, \hb_i\}_{i=1}^{N-1} \cup \{\hc_N\}$ are the ones defined by the shifted lift. Then

\begin{enumerate} \item[a.]  If $N = 3s+2$, we have
 \begin{equation}\label{ab1}\ha_k = \begin{cases} \alpha^{-1}\be a_k & k = 3r\\ \alpha^{-1}\be^{-2}a_k& k = 3r+1\\ \alpha^2 \beta a_k &k = 3r+2\end{cases}, \quad \hb_k = \begin{cases}  \alpha \be^2b_k& k = 3r\\ \alpha^{-2}\be^{-1} b_k & k= 3r+1\\ \al\be^{-1} b_k & k = 3r+2\end{cases}
\end{equation}
for $k =0, 1, 2, \dots, N-1$, and $\ha_N = a_N$, $\hb_N = \al^{-1} \be^{-2}b_N$, $\hc_{N+1} = \al \be^{-1} c_{N+1}$. 

\item[b.]   If  $N = 3s$, then
 \begin{equation}\label{ab2}\ha_k = \begin{cases} \alpha^2\be a_k & k = 3r\\ \alpha^{-1}\be a_k& k = 3r+1\\ \alpha^{-1}\be^{-2} a_k &k = 3r+2\end{cases}, \quad \hb_k = \begin{cases} \alpha\be^{-1} b_k& k = 3r\\ \alpha\be^2 b_k & k= 3r+1\\ \al^{-2}\be^{-1} b_k & k = 3r+2\end{cases}
\end{equation}
for $k = 0, 2, \dots, N-1$, and $\ha_N = \al\be^2 a_N$, $\hb_N = b_N$, $\hc_{N+1} = \al^{-1}\be^{-2} c_{N+1}$. 
\end{enumerate}
In both cases $\al$ and $\be$ are as in (\ref{albe1})-(\ref{albe2}).
\end{lemma}
\begin{comment} {\rm Notice that the relevant proportions that appear above
\begin{equation}\label{factors}
\al^{-1}\be = c_N, \quad\al^2\be = \frac{A_1}{A_0A_3}= \frac{A_1^2a_0}{A_3a_{N-1}c_N}, \quad \al\be^2 = \frac{c_N A_1}{A_3 A_0}= \frac{A_1^2a_0}{A_3a_{N-1}},
\end{equation}
are all rational expressions of the generating invariants.}
\end{comment} 
\begin{proof}
The proof of this lemma is a straightforward careful account of the different cases using the definitions
\[
a_i = \frac{\det(V_i, V_{i+1}, V_{i+3})}{\det(V_i, V_{i+1}, V_{i+2})},\quad b_i = \frac{\det(V_i, V_{i+3}, V_{i+2})}{\det(V_i, V_{i+1}, V_{i+2})}, \quad c_i = \frac{\det(V_{i+1}, V_{i+2}, V_{i+3})}{\det(V_i, V_{i+1}, V_{i+2})}\]
and the results of the previous lemma.
\end{proof}

The main theorem of this section (and of the paper) is now a consequence of the formulas found in (\ref{ab1})-(\ref{ab2}) above, and of the definition of $\al$ and $\be$.

\begin{theorem}\label{scalingth}
Consider the action of the one parameter group 
\begin{equation}\label{scaling}
a_k \to \mu a_k; \quad b_k \to \mu^{-1} b_k; \quad c_N \to c_N
\end{equation}
$k=0,\dots, N-1$, defined on the coordinates of a twisted pentagram spiral. Then the shift map is invariant under this action.
\end{theorem}

\begin{proof} First of all, the shift map in local coordinates is given by  
\[
\Smap(a_0, a_1, \dots, a_{N-1}, b_0, b_1, \dots, b_{N-1}, c_N) = (\ha_1, \ha_2,\dots, \ha_{N-1}, \ha_N, \hb_1, \hb_2, \dots,\hb_N, \hc_{N+1}).
\]
Using the results of our previous lemma, we first need to show that $\al$ and $\be$ are invariant under the shift, since $\ha_k$, $\hb_k$, $k=0,\dots N$, and $\hc_{N+1}$ are, in all cases, proportional to $a_k$, $b_k$ and $c_{N+1}$, respectively, with proportionality constants  given by different powers of $\al$ and $\be$. Since $\al$ and $\be$ are uniquely defined by equation (\ref{factors}), it suffices to show that 
\[
c_N \quad \hbox{\rm and}\quad\frac{A_3a_{N-1}}{A_1^2a_0}
\]
are invariant under the action, which they clearly are.

Finally, we need to study the formulas for $a_N$, $b_N$ and $c_{N+1}$ and show that the extended induced action on these invariants is given by 
\[
a_N \to \mu a_N, \quad b_N \to \mu^{-1} b_N, \quad c_{N+1} \to c_{N+1}.
\]
 Since $\ha_N$, $\hb_N$ and $\hc_{N+1}$ are multiples of $a_N$, $b_N$ and $c_{N+1}$ with factors given by different powers of $\al$ and $\be$, that would conclude the proof.

And indeed, we know that $a_N = a_1$, $b_N$ is determined by (\ref{rel2}) (in the appendix), which is invariant under the extended action. We also know that and  $c_{N+1} =\displaystyle \frac{c_N}{c_N+b_N a_{N-2}}$. The theorem is proved.
\end{proof}
\section{A non-standard Lax representation of the shift map}

We can finally build a Lax representation of the shift map using $\{\rho_i\}$. Indeed, consider the shift map  $\Smap: \TS(N,1)\to \TS(N,1)$ and let $\Smap$ be (abusing once more the notation) the natural lift of $\Smap$ to the polygons (or spirals, although they are not spirals in the traditional sense) with seed $\{V_i\}_{i=1}^{N+1}$ in $\R^3$. We can define $\Smap$ as either a map on the seed
\[
\Smap(\{V_i\}_{i=1}^{N+1}) = \{\hV_{i+1}\}_{i=1}^{N+1}
\]
or as a map on the vertices, $\Smap(V_i) = \hV_{i+1}$, for any $i$. Note that the map can be extended to be defined in all vertices using the pullback, since all vertices can we written in terms of the seed.  If we define it as a map on the vertices, then we can further extend the map to $\rho_i$ by defining
\[
\Smap(\rho_i) = (\hV_{i+1}, \hV_{i+2}, \hV_{i+3}) = \widehat\rho_{i+1},
\]
whenever $\rho_i = (V_i, V_{i+1}, V_{i+2})$. We can also extend it to algebraic combinations of $\rho_i$ as usual using the pull back, so that, for example, $\Smap(\rho_i^{-1}\rho_{i+1}) = \Smap(\rho_i)^{-1}\Smap(\rho_{i+1})$.

We define 
\[
K_i = \rho_i^{-1}\rho_{i+1}\quad\hbox{\rm and}\quad N_i = \rho_i^{-1} \Smap(\rho_i).
\]
Compatibility conditions imply
\begin{equation}\label{MCeq}
\Smap(K_i) = N_i^{-1} K_i N_{i+1}
\end{equation}
We have
\[
K_i = \begin{pmatrix}0&0&c_i\\ 1&0&b_i\\0&1&a_i\end{pmatrix} \quad\hbox{and}\quad\Smap(K_i) = \begin{pmatrix} 0&0&\Smap(c_i)\\ 1&0& \Smap(b_i)\\ 0&1& \Smap(a_i)\end{pmatrix},
\]
and therefore (\ref{MCeq}) above describes the map $\Smap$ in coordinates $\{a_i, b_i, c_i\}$.
We can also easily find $N_i$ explicitly. Let $r_i$ and $s_i$ be the powers of $\al$ and $\be$ such that
\[
S(V_i) = \hV_{i+1} = \al^{r_i}\be^{s_i} V_{i+1}
\]
as in lemma \ref{lemma1} (the values of $r_i$ and $s_i$, will depend on $i$ and are given in that lemma). Then
\begin{eqnarray*}
N_i &=& \rho^{-1}_i \Smap(\rho_i) = \rho^{-1}_i(\al^{r_i}\be^{s_i}V_{i+1},\al^{r_{i+1}}\be^{s_{i+1}}V_{i+2} ,\al^{r_{i+2}}\be^{s_{i+2}}V_{i+3} )\\&=& \rho_i^{-1}\rho_{i+1}R_i = K_i   R_i
\end{eqnarray*}
where
\[
R_i = \begin{pmatrix}  \al^{r_i}\be^{s_i}&0&0\\0&\al^{r_{i+1}}\be^{s_{i+1}}&0\\0&0&\al^{r_{i+2}}\be^{s_{i+2}}\end{pmatrix}
\]
and $\Smap(K_i) = R_i^{-1} K_{i+1} R_{i+1}$ is thus expressed as a shifted gauge by diagonal matrices (of course, the entries of $R_i$  depend on  the invariants, the map is highly non-linear). It is clear that $R_i$ is invariant under the scaling since $\al$ and $\be$ are. 

Consider now the modified $K_i(\mu), N_i(\mu)$ obtained through the introduction of the $\mu$-scaling
\[
K_i(\mu) = \begin{pmatrix} 0&0& c_i\\ 1&0&\mu^{-1} b_i\\ 0&1&\mu a_i\end{pmatrix}, \quad N_i(\mu) = K_i(\mu) R_i.
\]
Since $R_i$ is invariant under the $\mu$-scaling, if we substitute $K_i$ by $K_i(\mu)$ in (\ref{MCeq}), the resulting equation will not depend on $\mu$. Therefore, the system
\[
\rho_{i+1} = \rho_i K_i(\mu); \quad \Smap(\rho_i) = \rho_i N_i(\mu)
\]
$i = 0,1,2,\dots,N$,  is a discrete AKNS representation for the map $\Smap$. Here $a_N$ and $b_N$ are treated as functions of the generating system, their explicit expressions are found in (\ref{aN}) and (\ref{rel2}) in the appendix
\[
a_N = a_1, \quad b_N =  \frac{c_N}{a_{N-2}}\left(\frac{c_N}{A_1A_2}-1\right).
\]
In fact, this is not a classical Lax representation: the problem is neither periodic nor infinite, but rather it has different boundary conditions given by the spiral condition. Nevertheless, the  monodromy is clearly preserved from the construction of the map since the shifted spiral  has the same monodromy (this can also be double checked by straightforward calculations).  Once we fix coordinates $\G$ for twisted spirals, only the conjugation class of the monodromy is well defined. Indeed, on the one hand, from (\ref{rhoi}) in the appendix we know that
\[
\rho_{N+1} = M \rho_{-1}\A_1 = M \rho_0 K_{-1}^{-1}\A_1,
\]
where $\A_1$ is as in (\ref{ai1}), and also from the appendix
\[
a_{-1} = \frac{a_{N-1}^2a_{N-2}A_1A_2}{a_0^2c_NA_0}, \quad\quad b_{-1}= \frac1{a_0}\left(\frac{a_{N-1}c_N}{A_1a_0}-1\right) \quad\quad c_{-1} = \frac{a_{N-1}}{a_0}.
\]
On the other hand 
\[
\rho_{N+1} = \rho_0 K_0K_1\dots K_N.
\]
Therefore
\[
\rho_0^{-1} M \rho_0 = K_0K_1\dots K_N \A_1^{-1}K_{-1}.
\]
Recall that for the pentagram map the relation was similar, except for the matrix $\A_1^{-1}K_{-1}$, which appears here due to the different boundary conditions. Thus, assuming that $\mu$ and $r$ are complex numbers, the following theorem has already been proved
\begin{theorem} The eigenvalues of the matrix 
\[
M(\mu) = K_0(\mu)K_1(\mu)\dots K_N(\mu) \A_1^{-1}(\mu)K_{-1}(\mu)
\]
are preserved by the shift map, and the map preserves the Riemann surface
\begin{equation}\label{spec}
\det(M(\mu) - r I) = 0.
\end{equation}
\end{theorem}
 Notice that we have not found a Poisson structure  preserved by the shift map. Indeed, the structure associated to the pentagram map 
 \[
 \{a_k, a_{k\pm 3r}\} = \pm a_ka_{k\pm 3r}; \quad  \{b_k, b_{k\pm 3r}\} = \mp b_kb_{k\pm 3r}
 \]
 with all other brackets vanishing, is not preserved by the shift map even if we make $c_N$ a Casimir. This is a consequence of the fact that none of the factors in (\ref{factors}) commute with all $a_k$ or $b_k$ (implying that, for example, $\{S(a_k), S(b_r)\} \ne 0$ in general).  Still, we do have a good partial integration since many invariants of the map can be found analyzing the coefficients of the different powers of $\mu$ and $r$ in (\ref{spec}), which are all invariants. 
 
 Work in MAPLE shows that trace$(M(\mu))$ is a polynomial in $\mu$ containing only powers of $\mu$ in intervals of 3 (the same is true for the pentagram map). Likewise, not all powers appear in (\ref{spec}).  For example, if $N = 5$, (\ref{spec}) contains the following powers:
 \begin{eqnarray*}
&&\det(M(\mu) - r I) \\&=& I_0 + r \mu^{-7}I_1+r\mu^{-4} I_2+r^2\mu^{-2} I_3+r\mu^{-1}I_4+\mu^2 rI_5+\mu r^2I_6+\mu^{4}r^2I_7
 \end{eqnarray*}
where
\[
I_0 = 1, \quad I_1 = \frac{b_0b_1b_2b_3b_4}{a_4a_3a_0A_1^2A_2}(A_1A_2-c_5)(a_4c_5-a_0A_1) 
\]
\[
I_7 = a_1a_2a_3a_4, \quad I_5= \frac 1{A_1a_3c_5}(a_1a_3A_1^2+c_5^2a_4a_2)  
\]
\begin{eqnarray*}
I_3&=&\frac1{a_0}(b_1A_4+b_4(A_2+a_0b_2))-\frac1{A_1a_3}c_5(b_0+b_3A_1)\\
&-&\frac{A_1}{c_5a_4}(b_1+b_4A_2)+\frac{c_5^2}{A_1^2A_2a_3}(b_3A_1+b_0A_3)
\end{eqnarray*}
\[
I_6 = \frac 1{A_1A_2a_0}\left(a_0a_2c_5^2+a_4A_1A_2(a_3A_2+a_0(1+a_1b_3+a_3b_2))+A_1A_2a_0a_1(1+a_1b_4)\right)
\]
 while the remaining two are longer to write.
 
 One can directly see that $I_7$ is preserved. Indeed, from either (\ref{ab1}) or (\ref{ab2}) we have
 \[
 \Smap(a_1 a_2a_3a_4) = \ha_2\ha_3\ha_4\ha_5 = a_2a_3a_4a_5 = a_1a_2a_3a_4
 \]
 since $a_5 = a_N=a_1$.

  We have a final observation. In \cite{Schspiral} the author showed that the shift map could be viewed as the $(N+1)$th root of the pentagram map in a certain sense (notice that the domains are different). Recall (\cite{OST}) that in our notation the pentagram map is defined as
  \[
  T(e_i) = e_{i+2} \frac{[E_{i+2}, \dots, E_{i+N}]}{[E_{i-N+4}, \dots, E_{i-1}]}, \quad T(f_i) = f_{i-1}\frac{[E_{i-N},\dots,E_{i-5}]}{[E_{i+1},\dots,E_{i+N-4}]}
  \]
  for $N = 3m+2$, where $\{e_i,f_i\}$ are the twisted polygons coordinates we described in (\ref{penta}); and where the brackets $[\dots]$ denote the product of every three factors, $[E_2,\dots, E_N] = E_2E_5\dots E_{N-3}E_N$, with $E_i = 1+e_i f_{i-1}$ (same definition as $A_i$ but with polygons invariants instead of spiral invariants). This expressions are non-local, even though the map clearly is. The coordinates $e_i$ and $f_i$, (denoted by $a_i$ and $b_i$ in \cite{OST}), are mostly multiples, but not equal, to the spiral coordinates. As before, they have in common most of the normalizations that define them, except for beginning and end. The study here shows that in $a_i$, $b_i$ coordinates the spiral is local around the first vertex. This seems to indicate that the non-local terms in the pentagram map formula are the accumulation of $N+1$ factors, each one local, produced by the spiral.  This is reinforced by the fact that 
  \[
  \Smap(A_i) = A_{i+1}
  \]
  for $i=0,1,\dots,N-2$, and the fact that consecutive applications of $\Smap$ are calculated by shifting and multiplying by factors that are the $\Smap$ image of $c_N$, $\displaystyle\frac{A_1}{A_3A_0}$ and $\displaystyle\frac{A_3A_0}{A_1c_N}$, in succession. To make this statement precise one would need a careful study of the relation between the coordinates in the different moduli spaces, a calculation similar to the one in the proof of our appendix lemma, and one that we prefer not to include here. 
  
  Notice also that one can attempt to change the coordinates from $\{a_i, b_i$, $i=0,\dots,N-1$, $c_N\}$ to $\{A_i = 1+a_i b_{i-1}$, $B_i = 1+b_ia_{i-2}$, $i=0,\dots,N-1$, $c_N\}$, which might be possible for some values of $N$. In these coordinates the map looks like
  \[
  \Smap(A_i) = A_{i+1}, \quad \Smap(B_i) = B_{i+1}
  \]
  for $i=0,\dots,N-2$, but the images of $A_{N-1}$ and $B_{N-1}$ are not only non-local, but highly complicated and it does not seem to shed much light to the problem, we will not include further details.
  \appendix
  \section{Proofs}
  \subsection{Proof of theorem \ref{invariants}}
 First of all, let us write formulas for the liftings $T(V_i)$ and $\T(V_i)$.

From the lift definition
\begin{eqnarray*}
T(V_i) &=& (V_{i-1}\times V_{i+1})\times (V_i\times V_{i+2}) \\&=& \left(V_{i-1}\times V_{i+1}\right)\times(V_i\times (a_{i-1}V_{i+1}+b_{i-1}V_i+c_{i-1}V_{i-1}))
\\
&=&\det\rho_{i-1} (a_{i-1}V_{i+1}+c_{i-1}V_{i-1})
\end{eqnarray*}
for $i=0,1,\dots, N$, where we have used the rule $(a\times b)\times(b\times c) = \det(a,b,c) b$. From condition (\ref{normalization}) we see that $\det \rho_i = 1$, $i=0,\dots, N$. 

Notice that, from (\ref{spiralrel}) one also obtains
\begin{equation}\label{tform}
T(V_i) =\det\rho_{i-1} (a_{i-1}V_{i+1}+c_{i-1}V_{i-1})  =  \det\rho_{i-1}(V_{i+2} -b_{i-1} V_{i}) 
\end{equation}
for any $i=0,1,2,\dots,N$.

Analogously we can write a formula for $\T(V_i)$. Indeed, from the lift definition
\begin{eqnarray*}
\T(V_i) &=& c_{i+1}\left(V_i\times V_{i+1}\right)\times\left(V_{i-2}\times V_{i-1}\right) \\&=& c_{i+1}\left(V_i\times (a_{i-2}V_i+b_{i-2}V_{i-1}+c_{i-2}V_{i-2})\right)\times\left(V_{i-2}\times V_{i-1}\right)\\ &=& c_{i+1}\det\rho_{i-2}(b_{i-2}V_{i-1}+c_{i-2}V_{i-2}),
\end{eqnarray*}
which, as before, can be rewritten as
\begin{equation}\label{tbform}
\T(V_i) = c_{i+1}\det\rho_{i-2}(b_{i-2}V_{i-1}+c_{i-2}V_{i-2}) = c_{i+1}\det\rho_{i-2}(V_{i+1}- a_{i-2}V_{i}).
\end{equation}

The first point to notice is that the set $\{\ag_i, \b_i, \cg_i\}_{i=-\infty}^\infty$ determines the spiral up to a projective transformation. This is clear since we can write

\begin{equation}\label{spiralSF}
\rho_{i+1}= \rho_i \begin{pmatrix}0&0& \cg_i\\ 1&0&\b_i\\ 0&1&\ag_i\end{pmatrix} = \rho_i K_i
\end{equation}
for any $i$, where $\rho_i = (V_i, V_{i+1}, V_{i+2})$, and hence $\rho_0$ and $\{\ag_i, \b_i, \cg_i\}_{i=-\infty}^\infty$ determine the spiral. Therefore, to prove the theorem we need to show that  $\{\ag_i, \b_i, \cg_i\}_{i=-\infty}^\infty$ is generated only by $\G = \{\{\ag_i, \b_i\}_{i=0}^{N-1}, c_{N}\}$. 

We know that $\cg_i = 1$ for $i=0,1,2,\dots, N-1$. Also, since $\rho_{N+1} = \rho_N K_N$ and $\det K_N = c_N$, we have that
\[
c_N = \det \rho_{N+1}.
\]
The following lemma will show that both $K_{N+i}$ and $K_{-i}$ are shifted gauge transformations of lower and upper indexed invariants, respectively.
\begin{lemma}
For $i = 1,2,\dots$ 
\begin{equation}\label{gauge1} 
K_{N+i} = \A_i^{-1} K_{i-2}\A_{i+1},
\end{equation}
where $\A_i$ is local and depends only on $a_j, b_j, c_j$, $j=i, i-1, i-2$. We also have 
\begin{equation}\label{gauge2}
K_{-i} = \B_{-i}^{-1} K_{N-i-1}\B_{-i+1}
\end{equation}
for any $i=4,\dots, N$, with $\B_{-i}$ also local depending only on $a_j, b_j$ with $j=N-i, N-i-1$, $b_{N-i+1}$ and $c_r$, $r=N-i+4, \dots, N-i-1$. 
\end{lemma}
\begin{proof}

 Let's denote $d_i = \det \rho_i$. From (\ref{tform}), if $i = 1,2,\dots,N$, we have

\[
T(V_i) = d_{i-1}(a_{i-1}V_{i+1} + c_{i-1}V_{i-1}) = d_{i-1}\rho_{i-1}\begin{pmatrix}c_{i-1}\\0\\ a_{i-1}\end{pmatrix}
\]
and, since $\rho_{i-1} = \rho_{i-2} K_{i-2}$, with $K_i$ as in (\ref{spiralSF}), we can write
\[
T(V_i) = d_{i-1}\rho_{i-2} K_{i-2}\begin{pmatrix}c_{i-1}\\0\\ a_{i-1}\end{pmatrix} =d_{i-1}\rho_s K_sK_{s+1} \dots K_{i-2}\begin{pmatrix}c_{i-1}\\0\\ a_{i-1}\end{pmatrix},
\]
For any $s$. This is true for $i = 0,1,2,\dots$. Using these relations and straightforward calculations
\begin{eqnarray*}
\rho_{N+i+1} &=& M(T(V_{i}), T(V_{i+1}),T(V_{i+2}))\\ &=&M\rho_{i-1}\A_{i+1}
\end{eqnarray*}
$i=0,1,2,\dots$, where
\begin{equation}\label{ai1}
\A_{i+1}= \begin{pmatrix} d_{i-1}\begin{pmatrix}c_{i-1}\\0\\ a_{i-1}\end{pmatrix}& d_i K_{i-1} \begin{pmatrix}c_i\\0\\ a_{i}\end{pmatrix}& d_{i+1}K_{i-1}K_i\begin{pmatrix}c_{i+1}\\0\\ a_{i+1}\end{pmatrix}\end{pmatrix}.
\end{equation}
(Recall that  $V_{N+1}= M T(V_0)$, but $V_0 \ne M^{-1} \T(V_{N+1})$.)

We also have
\[
\rho_{N+i+1}=\rho_{N+i} K_{N+i}
\]
and as above
\begin{equation}\label{rhoi}
\rho_{N+i} =M\rho_{i-2}\A_{i}
\end{equation}
for $i = 1,2,3,\dots, N$. Putting these three relations together we get that $K_{N+i}$ is a gauged transformation of $K_{i-2}$ by the matrix $\A_i$
\[ 
K_{N+i} = \A_i^{-1} K_{i-2}\A_{i+1},
\]
$i = 1,2,\dots$.

A very similar argument can be made for negative indices. Indeed, $V_{-i} = M^{-1}\T(V_{N-i+1})$, $i=1,2,\dots$, where
\[
\T(V_i) =  c_{i+1}d_{i-2}(b_{i-2}V_{i-1}+c_{i-2}V_{i-2}) = c_{i+1}d_{i-2}\rho_{i-2}\begin{pmatrix}c_{i-2}\\ b_{i-2}\\0\end{pmatrix}.
\]
As before, 
\begin{equation}\label{rhomi}
\rho_{-i} = M^{-1}(\T(V_{N-i+1}), \T(V_{N-i+2}), \T(V_{N-i+3})) = M^{-1}\rho_{N-i-1}\B_{-i}
\end{equation}
for any $i\ge 3$, where
\begin{multline}\label{Bmi}
\B_{-i} =
\left(\begin{matrix}
c_{N-i+2}d_{N-i-1}\begin{pmatrix} c_{N-i-1}\\ b_{N-i-1}\\ 0\end{pmatrix}
& c_{N-i+3}d_{N-i}K_{N-i-1}\begin{pmatrix} c_{N-i}\\ b_{N-i}\\ 0\end{pmatrix}
\end{matrix}\right.  \\
\left.\begin{matrix}
& c_{N-i+4}d_{N-i+1}K_{N-i-1}K_{N-i}\begin{pmatrix} c_{N-i+1}\\ b_{N-i+1}\\ 0\end{pmatrix}
\end{matrix}\right).
\end{multline}
And since 
\[
\rho_{-i+1} = M^{-1}\rho_{N-i}\B_{-i+1}, \quad \rho_{-i+1} = \rho_{-i} K_{-i}
\]
for $i\ge 4$, we have that 
\[
\rho_{N-i}\B_{-i+1} = \rho_{N-i-1}K_{N-i-1}\B_{-i+1} = \rho_{N-i-1}\B_{-i} K_{-i}.
\]
Subsequently, the matrix $K_{-i}$ is a gauge transformation of $K_{N-i-1}$ by the matrix $\B_{-i}$
\[
K_{-i} = \B_{-i}^{-1} K_{N-i-1}\B_{-i+1}
\]
for any $i=4,\dots, N$. 
\end{proof}

\begin{proof} {\it of theorem \ref{invariants}}. 
The lemma allows us to conclude that the higher indexed invariants $K_{N+i}$, $i=1,2, 3, \dots, N$, can be written in terms of  $\G$, $K_N$ and $K_{-1}$, while those with negative subindex  can be written as functions of $\G$ and $K_{-1}$. In particular, straightforward calculations (these calculations are double checked with the computer) show that $K_{N+1} = \A_1^{-1}K_{-1}\A_2$ is given by
\begin{equation}\label{Kn}
K_{N+1} = \begin{pmatrix}\displaystyle0&0& c_{N+1}\\ 1&0& b_{N+1}\\ 0&1& a_{N+1}\end{pmatrix} = \begin{pmatrix}\displaystyle 0&0&\displaystyle \frac{c_{-1} A_2}{A_0}\\ 1&0&\displaystyle \frac{b_{-1} A_2}{A_0}\\ 0&1&a_2\end{pmatrix}
\end{equation}
where $A_i = c_i + a_i b_{i-1}$ for any $i$.

We are now left with proving that $K_{-3}$, $K_{-2}$, $K_{-1}$, $a_N$ and $b_N$ are functional combinations of  $\G$, and we will have proven the theorem. 
We start with the simplest ones, the matrices $K_{-2}$ and $K_{-3}$.
From (\ref{rhomi}) we see that 
\[
\rho_{-3} = M^{-1}\rho_{N-4} \B_{-3}.
\]
But our formula changes for either $\rho_{-2}$ or $\rho_{-1}$; we have
\[
\rho_{-1} = ( M^{-1}\T(V_{N}), V_0, V_1), \quad \rho_{-2} = (M^{-1}\T(V_{N-1}), M^{-1}\T(V_N), V_0).
\]
We also have,
\[
\T(V_{N-1}) = c_{N} \rho_{N-3}\begin{pmatrix}c_{N-3}\\ b_{N-3}\\ 0\end{pmatrix}. 
\]
From (\ref{tbform}), and using (\ref{spiralSF}) and (\ref{rhoi})
\[
\rho_{N-3}= \rho_{N+1}K_N^{-1}K_{N-1}^{-1}K_{N-2}^{-1}K_{N-3}^{-1} = M\rho_{-1} \A_1K_N^{-1}K_{N-1}^{-1}K_{N-2}^{-1}K_{N-3}^{-1}. 
\]
Putting these formulas together with the formula for $\rho_{-2}$, we obtain
\[
\rho_{-2} = \rho_{-1} \C_{-2}
\]
where 
\begin{equation}
\C_{-2} =
\left(\begin{matrix} c_{N}\A_1K_N^{-1}K_{N-1}^{-1}K_{N-2}^{-1}K_{N-3}^{-1}\begin{pmatrix} c_{N-3}\\ b_{N-3}\\ 0\end{pmatrix}&
 e_1& e_2 \end{matrix}\right).
\end{equation}

Given that $\rho_{-1} = \rho_{-2} K_{-2}$, we conclude that $K_{-2} = \C_{-2}^{-1}$. Also, since $\A_1$ involves only $K_0$, $K_1$, $K_{-1}$, we conclude that $K_{-2}$ is a function of the generators $\G$, $K_N$ and $K_{-1}$. We can also use this same process to study $K_{-3}$. On the one hand, from (\ref{rhoi}) we have
\begin{eqnarray*}
\rho_{-3} &=& M^{-1}\rho_{N-4} \B_{-3} = M^{-1}\rho_{N+1}K_N^{-1}K_{N-1}^{-1}K_{N-2}^{-1}K_{N-3}^{-1}K_{N-4}^{-1}\B_{-3} \\&=& \rho_{-1}\A_1K_N^{-1}K_{N-1}^{-1}K_{N-2}^{-1}K_{N-3}^{-1}K_{N-4}^{-1}\B_{-3}.
\end{eqnarray*}
Also
\[
\rho_{-2} = \rho_{-3} K_{-3} = \rho_{-1}\C_{-2}
\]
and hence
\[
\C_{-2} = \A_1K_N^{-1}K_{N-1}^{-1}K_{N-2}^{-1}K_{N-3}^{-1}K_{N-4}^{-1}\B_{-3}K_{-3}
\]
which proves that $K_{-3}$ is a function of $\G$,  $K_N$ and $K_{-1}$.

We are now down to $K_{-1}$ and $K_{N}$. 
Our next step is to use the condition imposed on the seed of the twisted spiral, namely $p_{N+1} = MT(p_0)$ being in the segment joining $p_N$ and $Mp_1$. This means the lift $MT(V_0) = V_{N+1}$ is in the homogeneous plane containing $MV_1$ and $V_N$, and from here
\[
\det(V_{N+1}, MV_1, V_N) = \det(MT(V_0), MV_1, V_N)= 0.
\]
Expanding as before (we will skip the details since it is the same type of calculation as above), we obtain
\begin{equation}\label{aN}
a_N = a_1. 
\end{equation}
Working with the extra relation 
\[
\det(MV_0, V_N, V_{N-1}) = 0\]
(which we can verify by simply observing fig. 3), and after some manipulations of the type we did before, we get
\begin{equation}\label{cm1}
c_{-1} = \frac{a_{N-1}}{a_0}.
\end{equation}

As  before
\begin{eqnarray*}
\T(V_{N}) &=& c_{N+1}\rho_{N-2}\begin{pmatrix}1\\ b_{N-2}\end{pmatrix} = c_{N+1}\rho_{N+1}K_N^{-1}K_{N-1}^{-1}K_{N-2}^{-1}\begin{pmatrix}1\\ b_{N-2}\\ 0\end{pmatrix} \\ &=& c_{N+1} M\rho_{-1}\A_1 K_N^{-1}K_{N-1}^{-1}K_{N-2}^{-1}\begin{pmatrix}1\\ b_{N-2}\\ 0\end{pmatrix}
\end{eqnarray*}
while 
\[
\rho_{-1} = (M^{-1} \T(V_N), V_0, V_1) = \rho_0 K_{-1}^{-1}.
\]
Thus, 
\[
K_{-1}^{-1} e_1 = \begin{pmatrix}\displaystyle -\frac{b_{-1}}{c_{-1}}\\\displaystyle -\frac{a_{-1}}{c_{-1}}\\ \displaystyle\frac 1{c_{-1}}\end{pmatrix} = c_{N+1}K_{-1}^{-1} \A_1 K_N^{-1}K_{N-1}^{-1}K_{N-2}^{-1} \begin{pmatrix} 1\\ b_{N-2}\\ 0\end{pmatrix}.
\]
Straightforward calculations show that these three equations reduce to the two relations
\[
c_N = (c_N+ b_N a_{N-2}) c_{N+1};  \quad a_{-1} = c_{N+1}\frac{c_{-1}a_{N-2}(1+a_1b_{0})}{c_{N}} = \frac{c_{N+1}a_{N-1}a_{N-2}}{c_Na_0}(1+a_1b_0),
\]
Which solves for $c_{N+1}$ and $a_{-1}$ in terms of generators $\G$ and $b_N$. But from (\ref{Kn}) we get
\[
c_{N+1} = \frac{c_{-1} A_2}{A_0}, \]
and so, using $A_i = c_i + a_ib_{i-1}$
\begin{equation}\label{rel}
\frac{c_{-1}}{1+a_0b_{-1}} = \frac{c_N}{(c_N+b_Na_{N-2})(1+a_2b_1)}
\end{equation}
which together with (\ref{cm1}) solves for $b_{-1}$ in terms of $\G$ and $b_N$.

Finally, $\rho_{N+1} = \rho_N K_N$, so that 
\[
c_N  = \det \rho_{N+1} = \det\rho_{-1}\det \A_1 = \frac1{c_{-1}}{A_0A_1} = \frac{(1+a_1b_0)(1+a_0b_{-1})}{c_{-1}}
\]
and so, from here and (\ref{rel}) we have 
\begin{equation}\label{rel2}
c_{-1} = \frac1{c_N} A_0A_1; \quad\frac{c_{-1}}{1+a_0b_{-1}} = \frac{1+a_1b_0}{c_N} = \frac{c_N}{(c_N+b_Na_{N-2})(1+a_2b_1)}.
\end{equation}
which allow us to solve for $b_N$ in terms of $\G$ only.

\end{proof}
\subsection{Proof of lemma \ref{lemma1}}
\begin{comment}{\rm  The arbritary lifts $\{\tV_i\}$ that will be used in the proof of this lemma are only arbitrary for $i=1,2,\dots,N+1$, while $\tV_{N+i} = M\cdot \widetilde{T(V_{i-1})}$, $i=2,3,\dots$, is obtained by substituting $\tV_j$ in the definition of $T(V_{i-1})$ in the statement of theorem \ref{lift}. Likewise for $\tV_{-i} = M^{-1}\cdot \widetilde{\T(V_{N-i+1})}$, $i=1,2,\dots$. But $\tV_0$ will be different for the spiral and its shift since the seed of the shifted spiral does not include $p_1$, {\it which is the reason why  $\hV_{N+1} \ne M T(\hV_0)$, but rather $\hV_0 = M^{-1}\T(\hV_{N+1})$}, the same way we defined $V_{-1}$ for the unshifted spiral. Therefore, we will use $\tV_{N+1} = MT(\tV_0)$ for the spiral and $\tV_0 = M^{-1}\T(\tV_{N+1})$ for the shifted one.  }
\end{comment}

\begin{proof}
From theorem \ref{lift}, there are some proportions $\l_i$ and $\hat\l_i$ such that 
\[
V_i = \l_i \V_i,\quad i=0,\dots, N, \quad \hV_i = \hl_i \V_i,\quad i=1,\dots, N+1
\]
for some arbitrary common lift $\{\tilde V_i\}$.

The relation between $\l_i$ and $\hl_i$ can be found as follows: both of them satisfy the same normalizations  with $i = 2,3,\dots, N$, the difference being the extra equation added to these ($i=1$ for $\l_i$ and $i=N+1$ for $\hl_i$) and the different definitions of $\l_i$ at the boundary. Using {\it only the equations and definitions they have in common}, we can solve for all $\l_i$'s in terms of $\l_{N}$ and $\l_{N+1}$, and then use either $i=1$ or $i=N+1$ and the particular definitions at the boundaries to solve for $\l_N$ and $\l_{N+1}$ or $\hl_N$ and $\hl_{N+1}$, respectively. Although we can use general arguments to prove the theorem, we will need to know with precision how $\l_i$ and $\hl_i$ depend on $\l_N$, $\l_{N+1}$ and $\hl_N$, $\hl_{N+1}$, so we go into the details. 

 We define $\l_{N+i} = \l_{i-2}\l_{i-1}\l_i\l_{i+1}$ for $i \ge 1$ so that $V_{N+i} = \l_{N+i} \tV_{N+i}$ with 
\[
M^{-1} \tV_{N+i} = (\tV_{i-2}\times\tV_i)\times(\tV_{i-1}\times\tV_{i+1})
\]
Defined by this formula. We define also $\l_{-1}$ as the proportion satisfying $V_{-1} = \l_{-i}\tV_{-1}$ with $\tV_{-1}$ also defined using the definition of $\T$ in (\ref{TinvV}) on $\{\tV_i\}$.  We also define $\hl_{N+k} = \hl_{k-2}\hl_{k-1}\hl_k\hl_{k+1}$ for $k\ge 2$ with $\hl_{N+1}$ this time independent from other lambdas, and $\hl_0$ given by the relation $\hV_0 = \hl_0M^{-1}\T(\tV_{N+1})$. (Because of the shifting of the seed we have $V_{N+1} = MT(V_0)$, while $\hV_0 = M^{-1}\T(\hV_{N+1})$.)

 The values of $\l_i$  are determined by the equations 
\[
\det(V_i, V_{i+1}, V_{i+2}) = 1\]
for $i = 0,1,\dots,N$, and the proportions $\hl_i$ are determined by the equations
\[
\det(\hV_i, \hV_{i+1}, \hV_{i+2}) = 1
\]
for $i = 1,2,\dots, N+1$.

Thus, let $\l_i$, $i=0,\dots,N$ be proportions satisfying the common equations
\begin{equation}\label{common}
\l_i\l_{i+1}\l_{i+2} = g_i
\end{equation}
$i=1,2,3,\dots,N$, where $g_i^{-1} = \det(\V_i, \V_{i+1}\V_{i+2}$). 

{\it Case $N = 3s+2$}. If we divide consecutive equations, we have
\begin{eqnarray*}
\l_1&=&\l_4\frac{g_1}{g_2} = \dots = \l_{N+2}\frac{[g_1\dots g_{N-1}]}{[g_2\dots g_{N}]} = \l_0\l_1\l_2\l_3 \frac{[g_1\dots g_{N-1}]}{[g_2\dots g_{N}]} \\ &=& \l_0 g_1\frac{[g_1\dots g_{N-1}]}{[g_2\dots g_{N}]}\\
\l_2&=&\l_5\frac{g_2}{g_3} = \dots = \l_{N}\frac{[g_2\dots g_{N-3}]}{[g_3\dots g_{N-2}]} \\
\l_3&=&\l_6\frac{g_3}{g_4} = \dots = \l_{N+1}\frac{[g_3\dots g_{N-2}]}{[g_4\dots g_{N-1}]} \end{eqnarray*}
where the bracket  represents the product of every third function, as in $[g_2\dots g_N] = g_2g_5g_8\dots g_N$. We can also use the common relation $\l_1\l_2\l_3 = g_1$ to obtain
\begin{equation}\label{lNrel}
\l_0= \l_N^{-1} \l_{N+1}^{-1} \frac{g_N}{g_1};\quad \l_1 = \l_N^{-1}\l_{N+1}^{-1} H_1;\quad \l_2 = \l_N H_2;\quad \l_3 = \l_{N+1} H_3,
\end{equation}
where $H_i$ depend only on the lift $\{\tV_i\}$ with the exception of $\tV_0$, which is the lift that the spiral and its shift do not share.  Therefore, all $\l_i$ depend on $\l_N$ and $\l_{N+1}$ through their relation to $\l_1, \l_2, \l_3$ above.

From theorem \ref{lift} we know that $\l_N$ and $\l_{N+1}$ will be determined uniquely by the corresponding normalizations, and likewise for $\hl_N$ and $\hl_{N+1}$ (the explicit formulas can also be obtained through straightforward calculations).  

Assume next that 
\[
\hl_N = \al \l_N \hskip 3ex\text{and}\hskip 3ex \hl_{N+1} = \be\l_{N+1}. 
\]
Then, using (\ref{lNrel}) and (\ref{lambdas}) we can prove (\ref{al1})-(\ref{al12}) straightforwardly with some case-by-case consideration.  The value of $\alpha$ and $\be$ are  found through two relations, each one coming from one end of the seed: the first one is
\begin{eqnarray*}
1 &=& \det(\hV_{N+1}, \hV_{N+2}, \hV_{N+3})  = \hl_{N+1}\hl_{N+2}\hl_{N+3}\det(\tV_{N+1}, \tV_{N+2}, \tV_{N+3}).
\end{eqnarray*}
We have, from (\ref{lNrel})
\[
 \hl_{N+1}\hl_{N+2}\hl_{N+3}=\hl_{N+1}\hl_0\hl_1^2\hl_2^2\hl_3^2\hl_4 = \hl_{N+1}\hl_0\hl_1g_1g_2 = \hl_N^{-2}\hl_{N+1}^{-1} H
 \]
where $H$ depends only on the arbitrary lift (excluding $\tV_0$). Therefore,  undoing this reasoning and going back to the unshifted spiral, we have that 
 \[
1 =  \al^{-2}\be^{-1} \det(V_{N+1}, V_{N+2}, V_{N+3}) = \al^{-2}\be^{-1} \det\rho_{N+1}.
\]
But, as we saw before, $\rho_{N+1} = \rho_N K_N$ with $\det \rho_N = 1$, so $\det\rho_{N+1} = c_N$ as claimed.

The second equation comes from the following observation at the other end of the seed: from (\ref{tbform})
\[
V_0 = \l_0 \tV_0 = \al\be\hl_0\tV_0=\al\be \hV_0 = \al\be M^{-1}\T(\hV_{N+1}) = \al\be M^{-1}\left((\hc_{N+2}(\hb_{N-1}\hV_N +\hV_{N-1})\right),
\]  
but we also have
\begin{eqnarray*}
\hc_{N+2} &=& \frac{\hl_{N+5}}{\hl_{N+2}} \tc_{N+2} = \frac{\l_3\l_4\l_5\l_6}{\l_0\l_1\l_2\l_3 }\tc_{N+2}\\ 
\hb_{N-1} &=& \frac{\hl_{N+2}}{\hl_N} \tb_{N-1} = \frac{\hl_0\hl_1\hl_2\hl_3}{\hl_N}\tb_{N-1},
\end{eqnarray*}
which, implies $\hc_{N+2} = \al\be^2c_{N+2}$ and $ \hb_{N-1} = \al^{-2}\be^{-1} b_{N-1}$. Since $\hl_{N} = \al \l_N$ and $\hl_{N-1} = \al^{-1}\be^{-1} \l_{N-1}$, putting everything together we get
\begin{eqnarray*}
V_0 &=& \al\be^2 M^{-1}\left((c_{N+2}(b_{N-1}V_N +V_{N-1})\right)\\ &=& \al\be^2M^{-1}\T(V_{N+1}) = \al\be^2M^{-1}\T(MT(V_0)).
\end{eqnarray*}
 Thus, {\it $\al\be^2$ measures how the lift  $\T$ fails to be the twisted inverse of the lift of $T$}. But we also have from (\ref{tform}) and (\ref{tbform}) that
 \[
\T(V_{N+1}) =  \T(MT(V_0)) = c_{N+2} \left(MT(V_0)\times MT(V_1)\right)\times\left( V_{N-1}\times  V_N\right).
 \]
Now, $\rho_{N-1} = \rho_{N+1}K_N^{-1} K_{N-1}^{-1}$, and therefore
\begin{eqnarray*}
M^{-1} V_{N-1} &=& M^{-1} \rho_{N-1} e_1 = M^{-1}\rho_{N+1}K_N^{-1} K_{N-1}^{-1}e_1 \\ &=& (T(V_0), T(V_1), T(V_2)) \begin{pmatrix} -a_{N-1}+\frac{\displaystyle b_Nb_{N-1}}{\displaystyle c_N}\\\\ \frac{\displaystyle A_N}{\displaystyle c_{N}}\\\\ -\frac{\displaystyle b_{N-1}}{\displaystyle c_N}\end{pmatrix}.
\end{eqnarray*}
Likewise
\[
M^{-1} V_N =  (T(V_0), T(V_1), T(V_2))\begin{pmatrix}-\frac{\displaystyle b_N}{\displaystyle c_{N}}\\\\-\frac{\displaystyle a_N}{\displaystyle c_N}\\\\ \frac 1{\displaystyle c_N}\end{pmatrix}
\]
and hence, after long but straightforward calculations using $a_N = a_1$ and $c_{-1} =\displaystyle \frac{a_{N-1}}{a_0}$, we get
\[
M^{-1}V_{N-1}\times M^{-1} V_N = \frac{A_1A_0}{c_N} V_0\times V_1 - \frac{b_N A_0}{c_Nc_{-1}} V_2\times V_0.
\]
We also have, from (\ref{tform})
\[
T(V_0)\times T(V_1) = \frac{A_0}{c_{-1}} V_2\times V_0.
\]
Finally, also using (\ref{tform}) and straightforward calculations we get
\[
c_{N+2} = \frac{\det(T(V_2), T(V_3), T(V_4))}{\det(T(V_1), T(V_2), T(V_3))} = \frac{A_2A_3}{A_1A_2} = \frac{A_3}{A_1}.
\]
Putting everything together we have
\[
M^{-1}\T(M T(V_0)) = c_{N+2} \frac{A_0^2A_1}{c_{-1}c_N} \left(V_2\times V_0\right)\times\left(V_0\times V_1\right) = \frac{A_3A_0^2}{c_{-1}c_N} V_0
\]
and hence $\al^{-1}\be^{-2} = \frac{\displaystyle A_3A_0^2}{\displaystyle c_{-1}c_N}$ as stated. Notice that the computation above is the same for $N = 3s$, only the powers of $\al$ and $\be$ change.

These two relations determine $\al$ and $\be$ generically. Finally, since $\hV_i = \displaystyle\frac{\hl_i}{\l_i} V_i$, $i=1,\dots,N$, using (\ref{lNrel}) we prove the lemma. 

{\it Case $N = 3s$}. The proof in this case is identical, but we use different equations for $\l_i$. Indeed, the systems of equations derived from the normalizations are now given by
\begin{eqnarray*}
\l_1&=&\l_4\frac{g_1}{g_2} = \dots = \l_{N+1}\frac{[g_1\dots g_{N-1}]}{[g_2\dots g_{N}]} \\
\l_2&=&\l_5\frac{g_2}{g_3} = \dots = \l_{N+2}\frac{[g_2\dots g_{N-3}]}{[g_3\dots g_{N-2}]} =\l_0g_1\frac{[g_2\dots g_{N-3}]}{[g_3\dots g_{N-2}]}\\
\l_3&=&\l_6\frac{g_3}{g_4} = \dots = \l_{N}\frac{[g_3\dots g_{N-2}]}{[g_4\dots g_{N-1}]} 
\end{eqnarray*}
and equation $\l_1\l_2\l_3 = g_1$ becomes also
\[
\l_0 = \l_N^{-1}\l_{N+1}^{-1} \frac{g_N}{g_1}.
\]
When we put them together as before we get
\[
\l_0 = \l_N^{-1}\l_{N+1}^{-1}F_0; \quad \l_1 = \l_{N+1} F_1, \quad \l_2 = \l_N^{-1}\l_{N+1}^{-1}F_2; \quad \l_3 = \l_N F_3,
\]
where again $F_i$ depends only on the arbitrary lift excluding $\tV_0$. If 
\[
\hl_N = \al \l_N,\quad\quad \hl_{N+1} = \be \l_{N+1},
\]
 we obtain  the relations (\ref{al2}) and (\ref{al22}) following the same reasoning as in the previous case. One can also check that
 \[
 V_0 = \al^{-2}\be^{-1}M^{-1}\T(MT(V_0)).
 \]

\end{proof}


\begin{thebibliography}{9}
\bibitem{A} Adler  M., {\it On a Trace Functional for Formal Pseudo-differential Operators and the
Symplectic Structure of the KdV}, Invent. Math. {\bf 50}, 219-248 (1979).
\bibitem{GSTV} M. Gekhtman, M. Shapiro, S. Tabachnikov and A. Vainshtein, {\em Higher pentagram maps, weighted directed networks, and cluster dynamics}, Electron. Res. Announc. Math. Sci., 19 (2012), 1Ð17.
 \bibitem{GD}  Gel'fand I.M. \& Dickey L.A.,  {\it A family of Hamiltonian structures connected with integrable
nonlinear differential equations}, in {\it I.M. Gelfand, Collected papers} v.1, Springer-Verlag, (1987).
\bibitem{Glick} M. Glick, {\it The pentagram map and Y-patterns}, Advances in Mathematics 227 (2011), 1019-1045.
\bibitem{KS} B. Khesin, F. Soloviev. {\em Integrability of higher pentagram maps}, Mathematische Annalen357, no. 3 (2013): 1005-47. 
 \bibitem{MMW} E. Mansfield, G. Mari-Beffa and J.P. Wang. {\em Discrete moving frames and discrete integrable systems}, Foundations of Computational Mathematics, Volume 13, Issue 4 (2013), Page 545-582, 10.1007/s10208-013-9153-0.
  \bibitem{MW} G. Mari-Beffa, J.P. Wang. {\em Hamiltonian evolutions of twisted polygons in $\RP^n$}, Nonlinearity 26 (2013) 2515-2551.
\bibitem{MB1} G. Mar\'\i~Beffa, {\em On generalizations of the pentagram map: discretizations of AGD flows},  Journal of Nonlinear Science: Volume 23, Issue 2 (2013), Page 303-334.
\bibitem{MB2} G. Mar\'\i~Beffa, {\em On integrable generalizations of the pentagram map},
International Mathematics Research Notices (2014); doi: 10.1093/imrn/rnu044.
\bibitem{M2} G. Mar\'\i~Beffa, {\em On bi-Hamiltonian flows and their realizations as curves in real semisimple homogeneous manifolds}, Pacific Journal of Mathematics, {\bf 247-1} (2010), pp 163-188.
\bibitem{OST} V. Ovsienko, R. Schwartz and S. Tabachnikov, {\em The Pentagram map: a discrete integrable system}, Communications in Mathematical Physics {\bf 299}, 409-446 (2010). 
 \bibitem{OTS2} V. Ovsienko, R. Schwartz and S. Tabachnikov, {\em Liouville-Arnold integrability of the pentagram map on closed polygons},  Duke Math. J., 162 (2013), 2149Ð2196.
 \bibitem{S1}R. Schwartz, {\em The Pentagram Map is Recurrent},
Journal of Experimental Mathematics, (2001) Vol {\bf 10.4} pp. 519-528.
\bibitem{S2} R. Schwartz, {\em The Pentagram Integrals for Poncelet Families} 
Journal of Geometry and Physics (to appear).  
\bibitem{S3} R. Schwartz, {\em Discrete monodromy, pentagrams and the method of condensation}, J. of Fixed Point Theory and Appl. {\bf 3} (2008), 379-409. 
\bibitem{ST} R. Schwartz and S. Tabachnikov,  {\em The Pentagram Integrals for Inscribed Polygons} 
Electr. J. Comb., 18 (2011), P171. 
\bibitem{Schspiral} R. Schwartz, {\em Pentagram Spirals}, to appear in J. Exp Math. (2013).
\bibitem{heat} R. Schwartz, {\it The Projective Heat Map on Pentagons}, preprint (2014).
\bibitem{FS} F. Soloviev, {\em Integrability of the Pentagram Map}, Duke Math. Journal, vol. 162 (2013), no.15, 2815--2853.
\bibitem{Wi} E.J. Wilczynski, {\em Projective differential geometry of curves and ruled surfaces}, 
 B.G. Teubner, Leipzig, 1906.



\end{thebibliography}
\end{document}